\RequirePackage{fix-cm}

\documentclass[11]{article}       
\usepackage[margin=1in]{geometry}
\usepackage{graphicx}
\usepackage{mathptmx}      
\usepackage{amsmath}
\usepackage{amssymb}
\usepackage{amsthm}
\usepackage[ruled]{algorithm2e}
\usepackage{xcolor}

\newtheorem{thm}{Theorem}[section]

\newtheorem{cor}[thm]{Corollary}
\newtheorem{prop}[thm]{Proposition}

\theoremstyle{definition}
\newtheorem{example}[thm]{Example}
\newtheorem{definition}[thm]{Definition}

\newenvironment{ex}{\begin{example}}{%
  \end{example}\ignorespacesafterend
}

\newcommand{\ts}{\textstyle}
\newcommand{\reals}{\mathbb{R}}

\newcommand{\setin}{\subseteq}

\newcommand{\intd}{\,\mathrm{d}}

\begin{document}

\title{A Simultaneous Perturbation Weak Derivative Estimator for Stochastic Neural Networks}


\author{Thomas Flynn$^1$         \and
        Felisa V\'azquez-Abad$^2$ 
}


\date{%
  $^1$Brookhaven National Laboratory ~ ({\tt tflynn@bnl.gov})\\%
  $^2$Hunter College, City University of New York ~ ({\tt felisav@hunter.cuny.edu})
      }

\maketitle

\begin{abstract}
  In this {\color{black} paper} we study  gradient estimation for a network of nonlinear stochastic units known as the Little model.  Many  machine learning systems can be described as networks of homogeneous units, and the Little model is of a particularly general form, which includes as special cases several popular machine learning architectures. However, since a closed form solution for the stationary distribution is not known, gradient methods which work for similar models such as the Boltzmann machine or sigmoid belief network cannot be used.
 To address this we introduce a method to calculate derivatives for this system based on measure-valued differentiation and simultaneous perturbation.  This extends previous works in which gradient estimation algorithms were presented for networks with restrictive features like symmetry or acyclic connectivity. 


\end{abstract}
\section{Introduction}
Many computational models in machine learning can be described as networks of homogeneous units. Although each unit may be very simple, capable of only trivial logical or mathematical operations, the hope is that through a training procedure the interconnection of these units can be arranged in order that the overall network can perform useful tasks, such as classification or regression.  In machine learning, ``training'' means the optimization of the parameters of the model in order to minimize an average loss over available data.
If gradient-based methods are to be used for the optimization, then a fundamental step will be computing the derivative of an appropriate cost function with respect to some network parameters. The exact means to compute the derivatives will depend on the details of the network - including whether the units are deterministic or stochastic, whether they satisfy any smoothness parameters, and network properties like symmetry, or the presence of cycles. While certain combinations of the aforementioned features lend themselves to easy treatment from the perspective of gradient estimation, in other cases it is a difficult problem. In this work we show how measure-valued differentiation can be fruitfully applied to this problem, to yield gradient estimators for very general network structures. This extends previous works in which gradient estimation algorithm's were presented for networks with restrictive features like symmetry or acyclic connectivity.

The networks we are interested in are probabilistic and operate on a finite state space. Each unit can be in one of two states, $0$ or $1$, and if there are $n$ nodes in total, then the state space is $X =\{0,1\}^{n}$. In general for stochastic neural networks, at each time step one or more units may change their state. The particular update rule in the network studied here was first defined in \cite{little1974}, and hence we sometimes refer to it as the Little model. Let $\xi(1),\xi(2) \hdots$ be a sequence of noise vectors in $\reals^{n}$, with the entire collection 
$\{ \xi_i(t) ; i=1,\hdots,n, t=1,2,\hdots \}$ independent and distributed according to the logistic distribution. That is, the cumulative distribution function of $\xi_{i}(t)$ is 
$$F(\xi_{i}(t) < x) = \sigma(x) := \frac{1}{1+\exp(-x)}.$$

{\color{black} The parameters of the model are denoted by $\theta \in \Theta \subset \reals^{n\times n} \times \reals^n $, where $\theta=(w,b)$ comprises the weights $w_{i,j}$ (in network terminology, these are the values of the corresponding links $(i,j)$) and the bias $b_i$ associated with  node $i$. 
}
Define
$
f : 
  X\times\Theta\times \Xi 
  \to
  X
$
 as
\begin{equation}\label{rtn-f}
f_{i}(x,(w,b),\xi)
 = \begin{cases} 
   1 \text{ if } \sum\limits_{j=1}^{n}w_{i,j}x_j + b_{i} > \xi_{i}, \\
   0 \text{ otherwise.}
\end{cases}
\end{equation}
This function $f$ and the noise $\xi(1), \xi(2) \hdots$ determines the operation of the  network; from the initial point $x(0)$ the states follows the recursion 
\begin{equation}\label{rtn-evolve}
x(t+1)
 =
f(x(t),\theta,\xi(t+1))
\end{equation} to generate the next state.
They can be interpreted as threshold networks, where the thresholds are {\color{black} random} at each time step. 

We let $P_{\theta}$ be the Markov kernel corresponding to this recursion, and $P_{\theta}(x^0,x^{1})$
be the probability of going to state 
$x^{1} \in X$ 
from state 
$x^{0} \in X$.
The function 
$u_i(x)$, 
that determines the input to each node at the state $x$  is defined as
\begin{equation}\label{u-def}
u_{i}(x,(w,b)) = \sum\limits_{j=1}^nw_{i,j}x_{j} + b_{i}.
\end{equation}
{\color{black}The above equation describes the input ``flow'' to node $i$ in terms of the network model, adding the bias to the total flow from incoming nodes. }
Then
\begin{equation}\label{p-rtn-def}
P_{\theta}(x^{0},x^{1})
=
\prod_{i=1}^{n}
\sigma(u_{i}(x^0,\theta))^{x^{1}_{i}}
(1-\sigma(u_{i}(x^0,\theta)))^{1-x^{1}_{i}}.
\end{equation}
Alternatively,
we can use the following notation of \cite{neal}: for 
$x \in \{0,1\}$,
 define 
\begin{equation}\label{star-def}
x^{\dag} = 2x - 1.
\end{equation}
Using this, {\color{black} together with  the identity
$1-\sigma(x) = \sigma(-x)$,
an equivalent expression to (\ref{p-rtn-def}) when $x_i\in \{0,1\}$ is given by:}
\begin{equation}\label{rtn-transition}
P_{\theta}( x^{0}, x^{1} ) 
= 
\prod_{i=1}^{n}\sigma( (x_{i}^{1})^{\dag}u_{i}(x^0,\theta)).
\end{equation}


In practice, a user can compute with this stochastic network by fixing an input and iterating the update rule for a large number of steps before observing the network state. To ensure that the long-run statistical behavior of the network is independent of the initial conditions, we should establish ergodicity of the Markov kernel. {\color {black} In Section \ref{ergo-prop}   we show that} $P_{\theta}$ is ergodic and find the convergence rate of the Markov chain in terms of $\|\theta\|$ and $n$. Denote by $\pi_{\theta}$ be the stationary measure of this Markov chain.

Given a cost function $e:X\to\reals$, the optimization problem is then
\begin{equation}\label{zzx-nine}
\min_{\theta \in \Theta}\, \left[ J(\theta) :=  \int_{X}e(x)\intd\pi_{\theta}(x) \right]
\end{equation}
and
$\pi_{\theta}$ is defined as the solution to 
$\pi_{\theta}P_{\theta} = \pi_{\theta}$, for the Markov kernel $P_{\theta}$ defined in (\ref{p-rtn-def}).
In this work the focus is on how to compute the {\color{black} gradient $\nabla_\theta J$}.

\section{Related work}
Several stochastic neural networks on discrete state spaces have been studied, and their gradient estimation procedures are based on having closed form solutions for the resulting probability distributions. The  works \cite{hinton83,peretto1984collective,synchboltzgrad1991} {\color{black}assume specific } constraints on network connectivity - for instance symmetry, or prohibiting cycles.  
The earliest neural network models to be studied from the computational view were the deterministic threshold networks \cite{MCCULLOCH199099,Rosenblatt58theperceptron}. In this model, each unit senses the states of its neighbors, takes a weighted sum of the values, and applies a threshold to determine its next state (either on or off). For single layer versions of these networks, where the units are partitioned into input and output groups, with connections only from input to output nodes, the corresponding optimization problem can be solved by the perceptron algorithm \cite{Rosenblatt58theperceptron}. Any iterative algorithm for optimizing threshold networks has to address the credit assignment problem \cite{minsky1961steps}. This means that during optimization, the algorithm must identify which internal components of the network are not working correctly, and adjust those units to improve the output. The difficulty in solving the credit assignment problem for threshold networks with multiple layers prevents simple deterministic threshold models from being used in complex problems like image recognition. {\color{black} Although there is yet no universal method to train the deterministic threshold networks, recent  research focuses on specific types of networks that can be trained using gradient-descent optimization}.   For instance, one can abandon the threshold units, and work with units that have a smooth, graded, response such as the sigmoid neural networks \cite{rumelhart1986learning}. In this case methods of calculus are available to determine unit sensitivities. These new networks are still deterministic but now operate on a continuous state space. 

Another approach is to keep the space discrete but make the network probabilistic, and use the smoothing effects of the noise to obtain a model one can apply methods of calculus to. One can interpret the sigmoid belief networks in this way. These networks were introduced in \cite{neal} and so named because they combine features of sigmoid neural networks and Bayesian networks. In these networks, when  a unit receives a large positive input it is very likely to turn on, while a large negative input means the unit is likely to remain off. In fact, these networks can be interpreted as threshold networks with random thresholds. The use of the sigmoid function, which is the cumulative distribution function (CDF) of the logistic distribution, leads to an interpretation of a network with thresholds drawn from the logistic distribution. In \cite{neal}, the author derived formulas for the gradient in these networks, and showed how Markov chain Monte Carlo (MCMC) techniques can be used to implement gradient estimators. The networks studied in \cite{neal} had a feed-forward architecture, but one could also define variants that allow cycles among the connections.
 If the connectivity graph in the Little model is acyclic, then one obtains a model resembling the sigmoid belief network. We can enforce this by requiring $w_{i,j} = 0$ if $i<j$. 
In this way one is lead to the random threshold networks. In this case, one would be interested in the long-term average behavior of the network. Such a generalization would resemble the random threshold networks that are our focus. It would be interesting to obtain a gradient estimator for these new networks.

Another motivation to study general random threshold networks comes from the Boltzmann machine  \cite{hinton83,ackley1985learning,harmonium}. This is a network of stochastic units that are connected symmetrically. This means there is feed-back in the network, and the problem in these networks is to optimize the long-term behavior. 
The Boltzmann machine follows an update rule similar to (\ref{rtn-f}), except that the weights are constrained to be symmetric ($w_{i,j} = w_{j,i}$) and the nodes are updated one at a time, instead of all at once. This model was an important ingredient in many machine learning systems \cite{hinton2006fast,MultimodalDBM}.  The symmetry in the network, and the use of the sigmoid function to calculate the probabilities, leads to a nice closed form solution for the stationary measure in this model. Based on formulas for the stationary distribution, expressions for the gradient of long-term costs can be obtained, leading to MCMC based gradient estimators. However, if one changes the model, by for instance using non-symmetric connections, or changing the type of nonlinearity,
{\color{black} closed} formulas are no longer available. Instead, one winds up with a model of random threshold networks.
Specifically, a model like the Boltzmann machine is obtained if the weights are symmetric, meaning
$w_{i,j} = w_{j,i}$. Technically, if one puts a symmetry requirement on our threshold networks, one does not exactly recover the Boltzmann machine, but a variant known as the synchronous or parallel Boltzmann machine \cite{peretto1984collective}. The synchronous Boltzmann machine also has a known, simple, stationary distribution \cite{peretto1984collective}.
This provides another motivation for studying gradient estimation in the Little model.

In the case of networks where cycles are allowed, the work \cite{apolloni1991learning} considered gradient estimation in the finite horizon setting. Our interest is in the long-term average cost in networks that have general connectivity, where only knowledge of the transition probabilities is available. Methods such as forward sensitivity analysis cannot be used in this case, as they rely on the differential structure of the underlying state space. Instead, we propose an algorithm that computes descent directions based on simultaneous perturbation analysis and measure-valued differentiation (MVD). 
 
 
In machine learning, most of the optimization problems are solved either using a closed form of the stationary averages and using a deterministic gradient-based method, or using a stochastic gradient when the cost function is smooth and one can interchange derivative and expectation. However, this limits the class of models that can be used in practice. On the other hand, networks with less constraints may have a more powerful modeling capability but estimating their gradients is not as straight forward and so far, little has been done to apply more sophisticated gradient estimation techniques in part because the typical practical applications require very high dimensional gradients and the code must be very efficient to be of practical use. \\
 
The main contribution of this paper to the area of machine learning is the application of MVD to a feedback neural network in order to train the machine using gradient descent stochastic approximation. The main contribution of the paper to the area of gradient estimation is the description of a novel methodology to efficiently calculate high dimensional gradients. Our proposed method combines the main idea of directional derivatives that has enjoyed unprecedented success with SPSA using finite differences, and  MVD gradient estimation, which yields unbiased estimation and thus, bounded variance.
 
\section{Preliminaries}
Since we are interested in performance of Markov chains, we should introduce our ergodicity framework.

\subsection{Ergodicity Framework}
We will work with the total variation metric: The total variation distance between two probability measures $\mu_1,\mu_2$ on a discrete space $X$ is
$$d_{TV}(\mu_1,\mu_2) = \sup_{B \setin X}|\mu_1(B) - \mu_2(B)|$$
where the supremum is over all subsets of $X$. The abstract result we use is the following: 

\begin{prop}\label{contract-prop}
  Let a Markov kernel $P$ have the property that $\min_{x,y}P(x,y) \geq  \epsilon> 0$.
  Then $P$ is a contraction for $d_{TV}$: For any two measures $\mu_1,\mu_2$, 
  $$
  d_{TV}(\mu_1 P , \mu_2 P) \leq (1-\epsilon)d_{TV}(\mu_1,\mu_2).
  $$
  In particular, if $\pi$ is the stationary measure of $P$ then for any initial measure $\mu$,
  \begin{equation}\label{contract-p}
    d_{TV}(\mu P^{t},\pi) \leq (1-\epsilon)^{t}d_{TV}(\mu,\pi).
    \end{equation}
\end{prop}
For instance, see
Lemma 2.20 in \cite{pflug-book}.
That lemma says that the convergence rate is at least
$1 - \min_{i,j}\sum_k \min\{p_{i,k},p_{j,k}\}$, and it's clear from our assumption that this is greater than $1-\epsilon$.

Next let us recall some of the basic ideas from measure-valued differentiation and simultaneous perturbation.
\subsection{Measure-Valued Differentiation}
The idea of measure-valued differentiation is to express the derivative of an expectation as the difference of two expectations. Each of these expectations involves the same cost function, but the underlying measures are different. This enables simple, unbiased derivative estimators if these measures are easy to sample from. For simplicity, in the following definitions we will only consider the (relevant to our purposes) case of a finite state space $X$. The method was pioneered by the ``weak derivative'' estimation method of \cite{pflug1990line,pflug-mvd}, and extended to cover discontinuous, and unbounded cost functions in the introductory papers \cite{heidergott2008measure,mvdrandom}. 

\begin{definition}\label{mvd-def}
Consider a  measure 
$\mu_{\theta}$ on a finite state space $X$ that depends on a real parameter $\theta$. 
  The measure 
  $\mu_{\theta}$
  is said to be differentiable at $\theta \in \reals$ if there is a triple 
  $(c_{\theta},\mu^{+}_{\theta},\mu^{-}_{\theta})$,
  consisting of a real number $c_{\theta}$, and two probability measures 
  $\mu^{+}_{\theta},\mu^{-}_{\theta}$ on $X$
  such that, for any function $e: X \to \reals$,
  $$
  \nabla_{\theta}\mu_{\theta}(e)
  =
  c_{\theta}[\mu^{+}_{\theta}(e) - \mu^{-}_{\theta}(e)].
  $$
\end{definition}
An MVD gradient estimator consists of two parts: First, sample a random variable 
$Y^{+}$ distributed according to 
$\mu^{+}_{\theta}$, then sample a random variable 
$Y^{-}$ according to 
$\mu^{-}_{\theta}$, and finally form the estimate 
$\Delta^{MVD} = c_{\theta}[e(Y^{+}) - e(Y^{-})]$.
For more background see \cite{hvapt}.
 The concept of MVD can be extended from measures to Markov kernels, and then applied to derivatives of stationary costs \cite{pflug1990line,pflug-mvd,mvdrandom}.  
\begin{definition}
A Markov kernel 
$P_{\theta}$
that is defined on a finite space $X$ and which depends on a real parameter 
$\theta$ 
is said to be differentiable at $\theta$ if for each 
$x \in X$ there is a triple 
$(c_{\theta}(x),P^{+}_{\theta}(x,\cdot),P^{-}_{\theta}(x,\cdot))$
which is the measure-valued derivative of the measure $P_\theta(x,\cdot)$ at the parameter $\theta$ in the sense of Definition \ref{mvd-def}. 

\end{definition}

If the Markov kernel 
$P_{\theta}$
is ergodic with stationary measure
$\pi_{\theta}$,
then in certain cases it is possible to use the
$(c_{\theta},P^{+}_{\theta},P^{-}_{\theta})$
to compute the stationary derivatives 
$\nabla_{\theta}\pi_{\theta}(e)$ \cite{pflug-mvd}.
The procedure is shown in Algorithm \ref{mvd-der-algo}. The correctness is expressed in 
Theorem \ref{mvd-stat-diff}, a result proved in \cite{pflug-mvd}. It gives a condition on a Markov chain $P_{\theta}$ that guarantees the corresponding stationary costs $\pi_{\theta}(e)$ are differentiable, and establishes that the procedure presented in Algorithm \ref{mvd-der-algo} (see below) can be used to compute the derivative of the stationary cost. Note that we are recalling a simplified version of the Theorem in the case of a finite state space. Similar results on differentiability for discrete chains have been developed by other authors, sometimes leading to different computational methods; see related results in \cite{Cao1998,kirkland03conditioningproperties}.

\begin{algorithm}[H]
  \caption{MVD gradient estimation for Markov chains ($\theta\in\reals$)\label{mvd-der-algo}}
  \For{$t=0,1,\hdots, M^{0}-1 $}{
    Sample $x(t+1)$ from $P_{\theta}(x(t),\cdot)$
  }
  - Sample $x^{+}(0)$ from $P_{\theta}^{+}(x(M^0),\cdot)$

  - Sample $x^{-}(0)$ from $P_{\theta}^{-}(x(M^0),\cdot)$

  \For{$t=0,1,\hdots, M^{1}-1 $}{
    Using common random numbers,
    
    \hspace{1em} Sample $x^{+}(t+1)$ from $P_{\theta}(x^+(t),\cdot)$,

    \hspace{1em} Sample $x^{-}(t+1)$ from $P_{\theta}(x^-(t),\cdot)$.
  }
- Set $
\displaystyle \Delta^{MVD}
 = 
c_{\theta}(x(M^0))\sum\limits_{t=1}^{M^{1}}[e(x^{+}(t)) - e(x^{-}(t))]
$

\Return $\Delta^{MVD}$.
\end{algorithm}

\medskip

Observe that Algorithm \ref{mvd-der-algo} has two settings $M^0$ and $M^1$. To emphasize the dependence of the random variable $\Delta^{MVD}$ on these parameters, we use the notation $\Delta^{MVD}_{M^0,M^1}$.
\begin{thm}[\cite{pflug-mvd}]\label{mvd-stat-diff}
{\color{black} Let $\theta\in\reals$ be a scalar parameter of a family of  Markov kernels $P_\theta$, and let $\delta_x P_\theta$ denote the (conditional) probability measure $P_\theta(x,\cdot)$, given state $x \in X$. Assume that}
$(\delta_{x}P_{\theta})(e)$ 
is differentiable in $\theta$ for each cost function $e$.  
Suppose further that $P_{\theta}$ is a contraction on the space of probability measures in the sense of Inequality \ref{contract-p}.
Then the stationary cost $\pi_{\theta}(e)$ is differentiable, and Algorithm \ref{mvd-der-algo} can be used to estimate the derivatives. Specifically, if we let $\Delta_{M^0,M^1}^{MVD}$ be the output of the algorithm, then 
$$
\lim_{M^{0}\to\infty,M^{1}\to\infty}\mathbb{E}\left[\Delta_{M^0,M^1}^{MVD}\right]
=
\nabla_{\theta}\pi_{\theta}(e).
$$
\end{thm}
{\color{black}More general results on MVD for stationary measures can be found in \cite{heidergott2003taylor,mvdrandom}, including the extension to discontinuous, and unbounded cost functions.}

Observe that we call for the use of common random numbers in the above algorithm. This ensures that the variance of the estimate $\Delta^{MVD}$ can be bounded independently of the parameters $M^{1}$ and $M^{0}$ \cite[Theorem~4.36]{pflug-book}. This can be achieved through any coupling with a contraction property - in the case of a finite state space, using common random numbers is one way - but this is not the only way; see \cite{pflug-book} for details. 

For measure-valued differentiation, like finite differences, it seems that in order to estimate the full derivative for a system with $m$ parameters, $m$ {\color{black}parallel} simulations are required  although the variance characteristics are much more favorable compared to finite differences (see \cite{pflug-book}, Section 4.3). In finite differences, one must trade off bias for variance, but for MVD the variance can be shown to be bounded independently of the parameters $M^1,M^0$ which determine the bias.
\subsection{Simultaneous Perturbation}
One interesting solution to the {\color{black} prohibitive computational complexity} of estimating derivatives with finite differences in high-dimensions is known as simultaneous perturbation \cite{spall1992multivariate}. In this scheme, one picks a random direction $v$, and then approximates the directional derivative using stochastic finite differences. In this case only two simulations are needed for a system with $m$ parameters. Using random directions in stochastic approximation was also studied in \cite[Section~2.3.5]{kushner1978stochastic} and \cite{ermoliev}.

The variance issues of finite differences remain with this approach; in order to decrease the bias of the estimator, one has to deal with a larger variance.  For generating the directions, one possibility is to let $v$ be a random point on the hypercube $\{-1,1\}^{m}$, as suggested in \cite{spall1992multivariate}. For the theoretical analysis, it is important that the directions have zero mean, and that the random variable $\frac{1}{\|v\|}$ be integrable. The procedure is shown in Algorithm \ref{sp-der-algo}.\\

\begin{algorithm}[H]
  \caption{Simultaneous perturbation derivative estimation for Markov chains\label{sp-der-algo}}
- Initialize $\lambda$ to a small positive number.

- Generate a random direction $v$ from the measure 
\begin{equation}\label{dir-measure}
P(v) 
= 
\prod_{i=1}^{n}[\tfrac{1}{2}\delta_{-1}(v_i)+\tfrac{1}{2}\delta_{1}(v_i)]
\end{equation}
\For{$t=0,1,\hdots, M-1 $}{
  
   Sample $x^{+}(t+1)$ from $P_{\theta + \lambda v}(x^+(t),\cdot)$,

   Sample $x^{-}(t+1)$ from $P_{\theta - \lambda v}(x^-(t),\cdot)$.
  }
  - Set $\displaystyle\Delta^{SP} = \frac{e(x^+(M)) - e(x^-(M))}{2\lambda}v$

\Return $\Delta^{SP}$
\end{algorithm}

\medskip

{\color{black} In the above algorithm, the random direction is generated according to  Spall's original SPSA method, namely each component of the random vector is either $-1$ or $1$ with equal probability, and components are independent \cite[Section V-A]{spall1992multivariate}. Other ways to generate random directions have been studied further, and our method can easily be generalized in this way as well. }

This algorithm will form the basis for the gradient estimator we derive for discrete attractor networks below.
\section{Ergodicity of Little model}\label{ergo-prop}
 In this section we consider the ergodicity of  the Little model \cite{little1974}. 
 We will show that $P_{\theta}$ has a contraction property with respect to $d_{TV}$.
 
The contraction coefficient $\alpha$ will depend on the weights of the network; larger weights will lead to a worse bound on the convergence.
Let us calculate the lower bound $\epsilon$ required by Proposition \ref{contract-prop} for the transition probability 
$P_{\theta}(x_0,x_1)$. 
According to the product form of $P_{\theta}$, (\ref{rtn-transition}), it suffices to find an 
$\epsilon_{1}$ such that
$\sigma( (x_{i}^{1})^{\dag}u_{i}(x^{0},\theta)) \geq \epsilon_1$ 
and then we can set $\epsilon = \epsilon_{1}^{n}$.
Note that for any value of $x_{i}^{1}$, we have 
\begin{align*}
  \sigma\left( (x_{i}^{1})^{\dag}u_{i}(x^{0},\theta)\right) 
    &\geq \sigma\left(-|u_{i}(x^{0},\theta)|\right) \\
    &\geq \sigma\left(-\textstyle\sum\limits_{j=1}^n|w_{i,j}| - |b_i|\right) \\
    &\geq 
    \sigma\left(-\|w\|_{\infty} - \|b\|_{\infty}\right).
\end{align*}
where $\|w\|_{\infty}$ is the $\infty$-norm of the matrix $w$.\footnote{This norm for matrices is defined as as $\|w\|_{\infty} = \sup_{\|u\|_{\infty}=1}\|wu\|_{\infty}$, where for the vector $u$ and $wu$, the norm $\|\cdot\|_{\infty}$ is defined in the usual way.} Letting $\epsilon = \sigma\left(-\|w\|_{\infty} - \|b\|_{\infty}\right)^{n}$, results in the formula
$$
d_{TV}(\mu_1 P ,\mu_2 P) 
\leq 
(1 - \sigma(-\|w\|_{\infty} - \|b\|_{\infty})^{n})d_{TV}(\mu_1,\mu_2)$$
{\color{black} It follows} that the predicted speed of convergence decreases as the weights increase.

Given the contraction property, we now consider the differentiability of the stationary distribution. Firstly, we are dealing with a finite state space and a Markov kernel with smooth transition probabilities (due to the smoothness of $\sigma$). Hence we are in a simple setting and the differentiability is relatively easy to establish. For instance Lemma 4 in \cite{pflug-mvd} can be applied. The conclusion is that for any cost function $e$, the stationary expectation $\pi_{\theta}(e)$ is a differentiable function of $\theta$.\\

{\color{black} {\sc Remark:} When using gradient-descent or other updating methods to find optimal $\theta$, one must ensure that either $\theta$ stays always inside a compact set, or that the sequence of consecutive updates is tight. This will follow from the convexity of the cost function, from Lyapunov stability, or from truncations performed in practice when applying such algorithms. This in turn establishes uniform ergodicity and differentiability of the stationary measure  at each step of the updates.}

\section{Gradient Estimation}
Gradient estimation has been studied for closely related models, such as the Boltzmann machine and sigmoid belief networks, as we discuss below. The case of the Little model is somewhat more challenging since there is not a known closed-form solution for the stationary distribution. One has to focus on gradient estimation methods that only use the Markov kernel associated to the process. 

We  propose an estimator which combines features of SP and MVD. The algorithm generates a random direction, as in SP, and then uses measure-valued differentiation to approximate this directional derivative. 
 In this way one deals with a small number of simulations, as in SP, while avoiding the variance issues with finite differences. {\color{black} We call the method {\em simultaneous perturbation measure-valued differentiation}}. The only requirement is that one can compute the measure-valued derivative along arbitrary directions. After developing the method in an abstract setting, we will consider the method in the context of the Little model.
\subsection{SPMVD}

{\color{black}We introduce the directional derivatives, as usual. Following standard notation, the gradient of the expected value of the cost function $e$  under measure $\mu_\theta$ is denoted by $\nabla_\theta \mu_\theta(e)$ and it is a row vector with as many components as the dimension of $\theta$. Therefore, the projection of this vector along a direction $v$ (a column vector with same number of components) is the (scalar) inner product $\nabla_\theta \mu_\theta(e)   \: v$. }

\begin{definition}\label{mvd-dir} Let 
$\mu_{\theta}$
be a measure depending on an $m$-dimensional vector parameter $\theta$. 
Let $v \in \mathbb{R}^{m}$ be a direction. A triple 
$(c_{\theta,v},\mu^{+}_{\theta,v}, \mu^{-}_{\theta,v})$
 is called a \textit{measure-valued directional directional derivative at $\theta$ in the direction $v$} if for all $e:X\to\reals$,
$$
\nabla_\theta \mu_{\theta}(e)\: v 
=
c_{\theta,v}[\mu^{+}_{\theta,v}(e) - \mu^{-}_{\theta,v}(e)].
$$
Note that in this expression, the left hand side is the dot product between the $m$ dimensional vectors $\nabla_\theta \mu_{\theta}(e)\:$ and $v$. The three terms making up the expression on the right side are all real-numbers.
\end{definition}
In practice, one can try to calculate the MVD in direction $v$ as follows. By basic calculus,
$$
\nabla_\theta \mu_{\theta}(e)\: v 
=
\nabla_{\lambda}\big|_{\lambda=0}\,\mu_{\theta + \lambda v}(e).$$
Therefore, to find the MVD of $\mu_{\theta}$ in direction $v$ it suffices to find the {\color{black}usual} scalar, MVD for 
$\mu_{\theta + \lambda v}$ at $\lambda=0$. This is the approach in the following example.

\begin{ex}
Let $\mu_1,\mu_2,\hdots,\mu_m$ be $m$ probability measures, and for a vector parameter $\theta \in \reals^{m}$ define $\mu_{\theta}$ as
$$\mu_{\theta} = \sum\limits_{i=1}^{m}\frac{e^{-\theta_i}}{Z_{\theta}}\mu_i$$
where $Z_{\theta} = \sum\limits_{i=1}^{m}e^{-\theta_i}$.
For any function $e:X\to\reals$ and direction $v \in \reals^{m}$, then, we have
\begin{equation}\label{mu-multivar}
\mu_{\theta + \lambda v}(e) 
=
\sum\limits_{i=1}^{m}\frac{e^{-(\theta_i + \lambda v_i)}}{Z_{\theta + \lambda v_i}}\mu_i(e).
\end{equation}


Introduce the notation $\gamma_{+} = \max \{ 0, \gamma\}$ and $\gamma_{-} = \max\{0, -\gamma\}$, for the positive and negative part of a scalar $\gamma$, respectively.  Then we have the following identities:
$\gamma = \gamma_{+} - \gamma_{-}$ and $|\gamma| = \gamma_{+} + \gamma_{-}$.
Define
$K_{\theta,v} = \sum\limits_{j=1}^{m}e^{-\theta_j}|v_{j}|$.
Differentiating (\ref{mu-multivar}) at  $\lambda=0$, and doing some algebra, one can get the following representation for the directional derivative: 
$$
\nabla_\theta\mu_{\theta}(e)\: v = 
\nabla_{\lambda}\big|_{\lambda=0}\,\mu_{\theta+\lambda v}(e) = 
c_{\theta,v}
\left[
\sum\limits_{i=1}^{m}\alpha_{i,v} \mu_i(e) - \sum\limits_{i=1}^{m}\beta_{i,v} \mu_i(e)
\right]
$$
where
$$c_{\theta,v} = \frac{K_{\theta,v}}{Z_{\theta}},$$
$$\alpha_{i,v} = 
\frac{1}{Z_{\theta}K_{\theta,v}}
\left[(v_i)_{-} Z_{\theta} + \sum\limits_{j=1}^{m}e^{-\theta_i}(v_{j})_{+}\right],
$$
and 
$$\beta_{i,v} = 
\frac{1}{Z_{\theta}K_{\theta,v}}
\left[(v_i)_{+} Z_{\theta} + \sum\limits_{j=1}^{m}e^{-\theta_i}(v_{j})_{-}\right].
$$
Therefore the triple 
$\left(c_{\theta,v},
      \sum\limits_{i=1}^{m}\alpha_i\mu_i,
       \sum\limits_{i=1}^{m}\beta_i\mu_i
\right)$
is the measure-valued derivative of $\mu$ at $\theta$ in the direction $v$.

To generate samples from the mixture $\sum\limits_{i=1}^{m}\alpha_i\mu_i$ one needs to choose one of the components $i=1,..,m$ according to the mixture probabilities $\alpha_i, i=1,\hdots, m$, and then generate a sample from $\mu_i$.  Similarly for the negative measure. Thus, one  needs to generate two random samples from the underlying collection of measures to get an unbiased derivative estimate. This concludes the example.
\end{ex}
Defining the extension {\color{black}of directional derivatives} to Markov chains is straightforward. 
\begin{definition}\label{mvd-kernel-dir} A triple 
$(c_{\theta,v},P^{+}_{\theta,v}, P^{-}_{\theta,v})$
is a measure-valued derivative for the Markov kernel $P$ at $\theta$ in the direction $v$ if for each $x$, 
$(c_{\theta,v}(x),P^{+}_{\theta,v}(x,\cdot),P^{-}_{\theta,v}(x,\cdot))$ is an MVD at $\theta$ in the direction $v$ for the measure $P_{\theta}(x,\cdot)$ in the sense of Definition \ref{mvd-dir}.
\end{definition}
The gradient estimator for stationary costs proceeds by choosing a random direction, and applying the stationary MVD procedure of Algorithm \ref{mvd-der-algo}. The pseudocode is presented as Algorithm \ref{spmvd}. \\

\begin{algorithm}[H]
  \caption{Simultaneous perturbation measure-valued differentiation (SPMVD)\label{spmvd}}
- Generate a random direction $v$ according to the distribution $p$:
\begin{equation}\label{random-dir-dist}
p(v) 
= 
\prod_{i=1}^{m}[\tfrac{1}{2}\delta_{-1}(v_i)+\tfrac{1}{2}\delta_{1}(v_i)]
\end{equation}

-Let
$(c_{\theta,v},P^{+}_{\theta,v},P^{-}_{\theta,v})$
 be a measure-valued derivative for $P_{\theta}$ in the direction $v$.

-Pass $P_{\theta}$ and $(c_{\theta,v},P^{+}_{\theta,v},P^{-}_{\theta,v})$ to scalar MVD (Algorithm \ref{mvd-der-algo}), to obtain the number $\Delta^{MVD}$.

-\Return $\Delta^{MVD}v$.
\end{algorithm}

\medskip
Note that $\Delta^{MVD}$ is a real-number, and the algorithm output, $\Delta^{MVD}v$, is a vector, of the same dimensionality as $\theta$.
 This estimator will be applied to our motivating example, the random threshold networks. 
\subsection{Application to the Little model}
Let us now give the measure-valued directional derivatives for the Little model.
Fix a parameter $\theta.$ A perturbation in the parameter space is represented as a vector
$v\in\reals^{n\times n}\times \reals^{n}$,
where $v_{i,j}$ is a perturbation along the weight from $j$ to $i$ and $v_{i}$ is a perturbation along the bias at node $i$.


 Using definitions (\ref{u-def}), (\ref{star-def}), (\ref{rtn-transition}), and after some algebra, one can obtain the following expression for the directional MVD:
\begin{equation}\label{mvd-eqn}
\nabla_{\theta}
\left[\sum\limits_{x^1}e(x^1)P_{\theta}( x^0, x^{1})\right]v = 
c_{\theta,v}(x^0)\left[
  \sum\limits_{x^1}e(x^1)P^{+}_{\theta,v}(x^0,x^{1}) -
  \sum\limits_{x^1}e(x^1)P^{-}_{\theta,v}( x^0, x^1)
\right],\end{equation}
where 
$$
c_{\theta,v}(x) 
= 
\sum\limits_{i=1}^{n}|v_i|\sigma(u_{i}(x^0,\theta)) 
+ 
\sum\limits_{i=1}^{n}\sum\limits_{j=1}^{n}|v_{i,j}|x_{j}^{0}\sigma(u_{i}(x^0,\theta))),$$
\begin{equation}\label{little-pplus-def}
  \begin{split}
& P^{+}_{\theta,v}(x^{0},x^1) 
 = \\
& \frac{P_{\theta}( x^{0},x^1)}{c_{\theta}(x^0)}\sum\limits_{i=1}^{n}
\left[
  x_{i}^{1}
    \left((v_{i})_{+} + \sum\limits_{j=1}^{n}(v_{i,j})_{+}x_{j}^{0}\right) +
  \sigma(u_{i}(x^{0},\theta))
    \left((v_{i})_{-} + \sum\limits_{j=1}^{n}(v_{i,j})_{-}x_{j}^{0}\right)
    \right].
\end{split}
\end{equation}
and
\begin{equation}\label{little-pminus-def}
  \begin{split}
&P^{-}_{\theta,v}(x^{0},x^1) 
= \\
&\frac{P_{\theta}(x^{0},x^1)}{c_{\theta}(x^0)}
\sum\limits_{i=1}^{n}
\left[
  x_{i}^{1}
    \left((v_i)_{-} + \sum\limits_{j=1}^{n}(v_{i,j})_{-}x_{j}^{0}\right) 
    +
    \sigma(u_{i}(x^{0},\theta))
   \left((v_{i})_{+} + \sum\limits_{j=1}^{n}(v_{i,j})_{+}x_{j}^{0}\right)\right].
\end{split}
\end{equation}
See Section \ref{mvd-der-eqn} for a detailed derivation of Equation \ref{mvd-eqn}. 
This yields two Markov kernels $P^{+}_{\theta,v}$ and $P^{-}_{\theta,v}$, that depend not just on the parameter $\theta$ (as would be the case in scalar MVD) but also the direction of interest $v$. 

In order for SPMVD to be useful, there must be a practical procedure for running the Markov kernels $P^{+}_{\theta,v}$ and $P^{-}_{\theta,v}$.  
The Little networks operate on a large state space, of size $2^{n}$ when $n$ nodes are used, so this is not necessarily trivial. The original Markov kernel $P_{\theta}$ has a relatively simple structure, allowing each node to be updated independently (see Eqn. \ref{rtn-transition}), and one may hope for a similar situation with  the MVD pair $P_{\theta,v}^{+}, P_{\theta,v}^{-}$. 

To investigate this, let us consider the Markov kernel 
$P_{\theta,v}^{+}$. 
Fix an $x_0$ and a $\theta$. We want to see how to {\color{black}generate an $x^1$ sample } from 
$P_{\theta,v}^{+}(x^{0},\cdot)$. 

Define the following variables:

\begin{enumerate}
   \renewcommand{\theenumi}{\textbf{P}\roman{enumi}}
\item\label{d-def}
$d 
= 
\sum\limits_{i=1}^{n}
\sum\limits_{j=1}^{n}
\sigma(u_i(x^{0},\theta))(v_{i,j})_{-}\, x_{j}^{0}
+
\sum\limits_{i=1}^{n}
\sigma(u_i(x^{0},\theta))(v_{i})_{-}
$
\item
$a_i = (v_{i})_{+} + \sum\limits_{j=1}^{n}(v_{i,j})_{+}\, x_{j}^{0}$, \, $i=1,\hdots,n$,
\item
$\beta_i = \sigma(u_i(x^0,\theta))$, \, $i=1,\hdots,n$,
\item\label{c-def}
$c = c_{\theta,v}(x^{0})$.
\end{enumerate}
Each of these depends on $x_0, \theta$ and $v$. Then the probability 
$Q(x) = P_{\theta,v}^{+}(x^0,x)$ 
has the representation
\begin{align}
\begin{split}\label{the-form}
Q(x) &= 
\frac{1}{c}
\left(d + \sum\limits_{i=1}^{n}
x_{i}a_{i}
 \right)
\prod_{i=1}^{n}\beta_i^{x_i}(1-\beta_i)^{1-x_i}.
\end{split}
\end{align}
We will sample from this distribution by sequentially generating the random variables
$x_1,\hdots,x_n$.  First
we will calculate and sample from the marginal distribution 
$Q(x_{1})$. Then we sample from the conditional distribution 
$Q(x_{2} \mid x_1)$, followed by sampling from 
$Q(x_{3} \mid x_{1},x_{2})$ and so on until finally sampling from 
$Q(x_{n} \mid x_1,\hdots, x_{n-1})$. 
This is a standard technique for generating random vectors (see Section 4.6 of \cite{ross1990course}). As we shall see, it is feasible since the conditional probabilities are easy to compute. 

To start, note that by the definition of conditional probability, 
\begin{equation}\label{def-cond}
  Q\left(x_{k} \mid x_{k-1},\hdots,x_1\right) 
  = 
  \frac{Q(x_{k},x_{k-1}\hdots,x_1)}
       {Q(x_{k-1},\hdots,x_1)}
\end{equation}
Based on this equation, if we can compute the marginal probabilities quickly then we can compute the conditional probabilities quickly. Starting from the formula (\ref{the-form}), one can show that  for $x_1 \in \{0,1\}$,
\begin{equation}\label{little-p-one}
\begin{split}
Q(x_1) 
&= 
\sum\limits_{x_2\in\{0,1\}}\hdots\sum\limits_{x_n\in\{0,1\}}Q(x_1,x_2,\hdots,x_n)
\\
&=
\frac{1}{c}\:\beta_1^{x_1}(1-\beta_1)^{1-{x_1}}
\left( d + \alpha_1x_1 + \sum\limits_{k=2}^{n}\alpha_i\beta_i \right)
\end{split}
\end{equation}
and
for any 
$(x_{k},\hdots,x_1) \in \{0,1\}^{k}$,
\begin{equation}\label{little-p-gen}
Q(x_{k}, x_{k-1} ,\hdots, x_1) 
= 
\frac{1}{c}
    \left(
      d + \sum\limits_{i=1}^{k}\alpha_i x_i 
        + \sum\limits_{i=k+1}^{n}\alpha_i\beta_i
    \right)
    \prod_{i=1}^{k}\beta_i^{x_i}(1-\beta_i)^{1-x_i}.
\end{equation}
See Appendix \ref{mvd-p} for a derivation of (\ref{little-p-one}) and (\ref{little-p-gen}). Based on these formulas, we construct the sampling algorithm (Algorithm \ref{mvd-samp-algo}).  Using the identity (\ref{little-p-gen}), we can write (\ref{def-cond}) as 
\begin{equation}\label{simple-form}
Q(x_k = 1 \mid x_{k-1},\hdots,x_1) = 
  \frac{
   \left(
     d + \sum\limits_{i=1}^{k-1}\alpha_ix_i + \alpha_k + \sum\limits_{i=k+1}^n\beta_i\alpha_i
   \right)\beta
   }
   { 
     d + \sum\limits_{i=1}^{k-1}\alpha_ix_i + \sum\limits_{i=k}^n\beta_i\alpha_i
   }
\end{equation}
\begin{algorithm}[H]
  \caption{Sampling from a measure of the form (\ref{the-form})\label{mvd-samp-algo}}
  - Set 
  $\delta_1
  = 
  Q(x_1 = 1)$ via Equation (\ref{little-p-one}).

  - Set $x_1=1$ with probability $\delta_{1}$, otherwise $x_1=0$.

  \For{$k=2,\hdots,n$}{

     Set 
    $\delta_{k} = Q(x_k=1\mid x_{k-1},\hdots,x_1)$ via Equation  (\ref{simple-form}).

     Set $x_{k}=1$ with probability $\delta_{k}$, otherwise $x_k=0$.

   }
\Return $(x_1,\hdots,x_n)$.
\end{algorithm}

\medskip

The following proposition certifies the correctness of Algorithm 
\ref{mvd-samp-algo}. 
\begin{prop}
For any data $c,d, a_1,\hdots,a_n, \beta_1,\hdots,\beta_n$, the output of Algorithm \ref{mvd-samp-algo} is distributed as follows:
$$
  p(x_1,\hdots, x_n) = 
  \frac{1}{c}\left(d + \sum\limits_{i=1}^n x_i a_i\right)
  \prod_{i=1}^n\beta_i^{x_i}(1-\beta_i)^{1-x_i}.
$$
\end{prop}
\begin{proof}
This follows from the identity 
(\ref{def-cond}) 
and the computational formulas 
(\ref{little-p-one}), 
(\ref{little-p-gen}), 
whose correctness is established in the appendix.
\end{proof}

Analogous to the variables (\ref{d-def}) - (\ref{c-def}), we have definitions for $P^-$:
\begin{enumerate}
   \renewcommand{\theenumi}{\textbf{M}\roman{enumi}}
\item\label{md-def}
$d 
= 
\sum\limits_{i=1}^{n}
\sum\limits_{j=1}^{n}
\sigma(u_i(x^{0},\theta))(v_{i,j})_{+}\: x_{j}^{0}
+
\sum\limits_{i=1}^{n}
\sigma(u_i(x^{0},\theta))(v_{i})_{+}
$
\item
$a_i = (v_{i})_{-} + \sum\limits_{j=1}^{n}(v_{i,j})_{-}\: x_{j}^{0}$, \, $i=1,\hdots,n$,
\item
$\beta_i = \sigma(u_i(x^0,\theta))$, \, $i=1,\hdots,n$,
\item\label{mc-def}
$c = c_{\theta,v}(x^{0})$.
\end{enumerate}
\begin{cor}
Let $c,d,a_1,\hdots,a_n,\beta_1,\hdots,\beta_n$ be defined as in (\ref{d-def}) - (\ref{c-def}). 
Then the distribution of the output of Algorithm \ref{mvd-samp-algo} is 
$P_{\theta,v}^{+}(x_0,x_1)$. 

Alternatively, if the $c,d,a_1,\hdots,a_n,\beta_1,\hdots,\beta_n$ are defined as in (\ref{md-def}) - (\ref{mc-def}) then the distribution of the output of Algorithm \ref{mvd-samp-algo} is 
$P_{\theta,v}^{-}(x_0,x_1)$.
\end{cor}
One can use Algorithm \ref{mvd-samp-algo} to simulate the Markov chains 
$P_{\theta,v}^{+}$ and 
$P_{\theta,v}^{-}$ that are needed in the directional MVD algorithm. Two copies would be used - one with the data for 
$P_{\theta,v}^+$ and one for $P_{\theta,v}^-$.

The gradient estimation procedure is summarized below. It is the MVD gradient estimation algorithm (Algorithm \ref{mvd-der-algo}), customized to include random directions and the special sampling algorithm (Algorithm \ref{mvd-samp-algo}). \\

\begin{algorithm}[H]
  \caption{MVD gradient estimation for the Little model\label{mvd-der-little}}
  - Generate a random direction $v$ from the distribution (\ref{random-dir-dist}).
  \vspace{1em}

  \For{$t=0,1,\hdots, M^{0}-1 $}{
     Sample $x(t+1)$ from $P_{\theta}(x(t),\cdot)$.
  }

  \vspace{1em}
  - {\color{black} Let $x^0 = x(M)$}.
  
    \vspace{1em}
    
    - Calculate $c^+,d^+,\alpha_1^+,\hdots,\alpha_n^+,\beta_1^+,\hdots,\beta_n^+$ according to formulas (\ref{d-def})-(\ref{c-def}).

  - Run Algorithm \ref{mvd-samp-algo} using data $c^+,d^+,\hdots,$ to obtain $x^+(0)$.
 
  \vspace{1em}
  - Calculate $c^-,d^-,\alpha_1^-,\hdots,\alpha_n^-,\beta_1^-,\hdots,\beta_n^-$ according to formulas (\ref{md-def})-(\ref{mc-def}).

  - Run Algorithm \ref{mvd-samp-algo} using data $c^-,d^-,\hdots,$ to obtain $x^-(0)$.

  \vspace{1em}
  \For{$t=0,1,\hdots, M^{1}-1 $}{
    Using common random numbers,
    
    \hspace{1em} Sample $x^{+}(t+1)$ from $P_{\theta}(x^+(t),\cdot)$,

    \hspace{1em} Sample $x^{-}(t+1)$ from $P_{\theta}(x^-(t),\cdot)$.
  }
- Set $
\displaystyle \Delta^{MVD}
 = 
c_{\theta}(x(M^0))\sum\limits_{t=0}^{M^{1}}[e(x^{+}(t)) - e(x^{-}(t))].
$

\Return $\Delta^{MVD}v$.
\end{algorithm}
Observe that the algorithm uses a coupling by common random numbers - see the discussion immediately after Algorithm \ref{mvd-der-algo} for the discussion of why this is case, and alternatives.

The theoretical soundness of this procedure rests on Theorem \ref{mvd-stat-diff}. That theorem tells us that conditioned on the random direction $v$, the output $\Delta^{MVD}$ converges in expectation to the directional derivative of the stationary cost in direction $v$, as $M^{0}$ and $M^{1}$ tend to infinity. This is established more formally in the following Proposition.
\begin{prop}
  Let $\Delta^{MVD}_{M^0,M^1}$ and $v$ be as defined in Algorithm \ref{mvd-der-little}.
  Then
  $$\lim_{M^0\to\infty ,M^1 \to \infty}\mathbb{E}\left[\Delta^{MVD}_{M^0,M^1}v\right] = \nabla_{\theta} J(\theta).$$
\end{prop}
\begin{proof}
  By Theorem \ref{mvd-stat-diff}, we know that for any sampled direction $v$,
  \begin{equation}\label{eqnX}
    \lim_{M^0\to\infty ,M^1 \to \infty}\mathbb{E}[\Delta_{M^0,M^1}^{MVD} \mid v] = \nabla_{\theta}J(\theta) \: v.
    \end{equation}
  In what follows,  we will omit the notation indicating the dependence on
  $M^0, M^1$,
  and write
  $\Delta^{MVD}$
  instead. By conditioning on $v$, then,
  \begin{equation}\label{eqnY}
    \begin{split}
    \lim_{M^0 \to \infty,M^1 \to \infty}\mathbb{E}\left[\Delta^{MVD} v\right]
    &=
    \lim_{M^0\to\infty,M^1 \to \infty}
    \mathbb{E}\left[ \mathbb{E}\left[\Delta^{MVD} v \mid v \right] \right]
    \\
    &=
    \mathbb{E}\left[\lim_{M^0\to\infty,M^1 \to \infty}
      \mathbb{E}\left[\Delta^{MVD}  \mid v \right] v \right].
    \end{split}
  \end{equation}
  The interchange of the limit and expectation in the above equalities is allowed since $v$ is a discrete random variable that can only take values in the finite set $\{-1,1\}^{n}$.
  Continuing, then, we can combine Equation \ref{eqnX} with Equation \ref{eqnY} to see that
  \begin{equation}\label{eqnZ}
    \lim_{M^0\to\infty,M^1 \to \infty}
    \mathbb{E}
    \left[\Delta^{MVD}  v\right] 
    =
    \mathbb{E}
    \left[ (\nabla_{\theta}J(\theta) \: v) v\right].
  \end{equation}
  In terms of the Euclidean basis vectors $e_i$, we have $v = \sum\limits_{i=1}^{n}v_ie_i$, and
  \begin{equation}\label{eqnA}
    \begin{split}
  \mathbb{E}\left[(\nabla_{\theta} J(\theta) v)v \right]
    &=
    \mathbb{E}\left[\sum\limits_{i=1}^{n} (\nabla_{\theta} J(\theta) v)v_ie_i\right] \\
    &= 
    \mathbb{E}\left[
      \sum\limits_{i=1}^{n}\sum\limits_{j=1}^{n} \nabla_{\theta_j}J(\theta) v_j v_ie_i
      \right] \\
    &= 
    \mathbb{E}\left[\sum\limits_{i=1}^{n} \nabla_{\theta_i}J(\theta) v_i^2e_i\right] \\
    &=
    \sum\limits_{i=1}^{n}\nabla_{\theta_i}J(\theta) e_i = \nabla_{\theta}J(\theta).
    \end{split}
  \end{equation}
  In the third equality we used the independence of the $v_i$, and that $\mathbb{E}[v_i] = 0$ for all $i$. In the fourth inequality, we used that  $\mathbb{E}[v_i^2] = 1$.
  Combining \ref{eqnZ} with \ref{eqnA} yields the result.
\end{proof}
This result can be compared with Lemma 1 of \cite{spall1992multivariate}, which concerns the bias in the traditional SPSA gradient estimates from Algorithm \ref{sp-der-algo}. 

A complexity analysis reveals that running the algorithm requires $\mathcal{O}\left( (M^{0} + M^{1} + 1 )N^{2}\right)$ operations. Note that this assumes full connectivity of the underlying network, which results in $N^{2}$ parameters. This complexity follows from the fact that it takes $\mathcal{O}\left(N^{2}\right)$ operations to iterate the underlying chain for one step, and also that our intermediate step of preparing and running Algorithm 4 takes $\mathcal{O}(N^{2})$ operations. 

In the next section we empirically investigate the behavior of Algorithm \ref{mvd-der-little} when it used inside an optimization procedure.

\section{Numerical Experiments}



\begin{figure}[t]
\begin{center}
\includegraphics[width=0.5\linewidth]{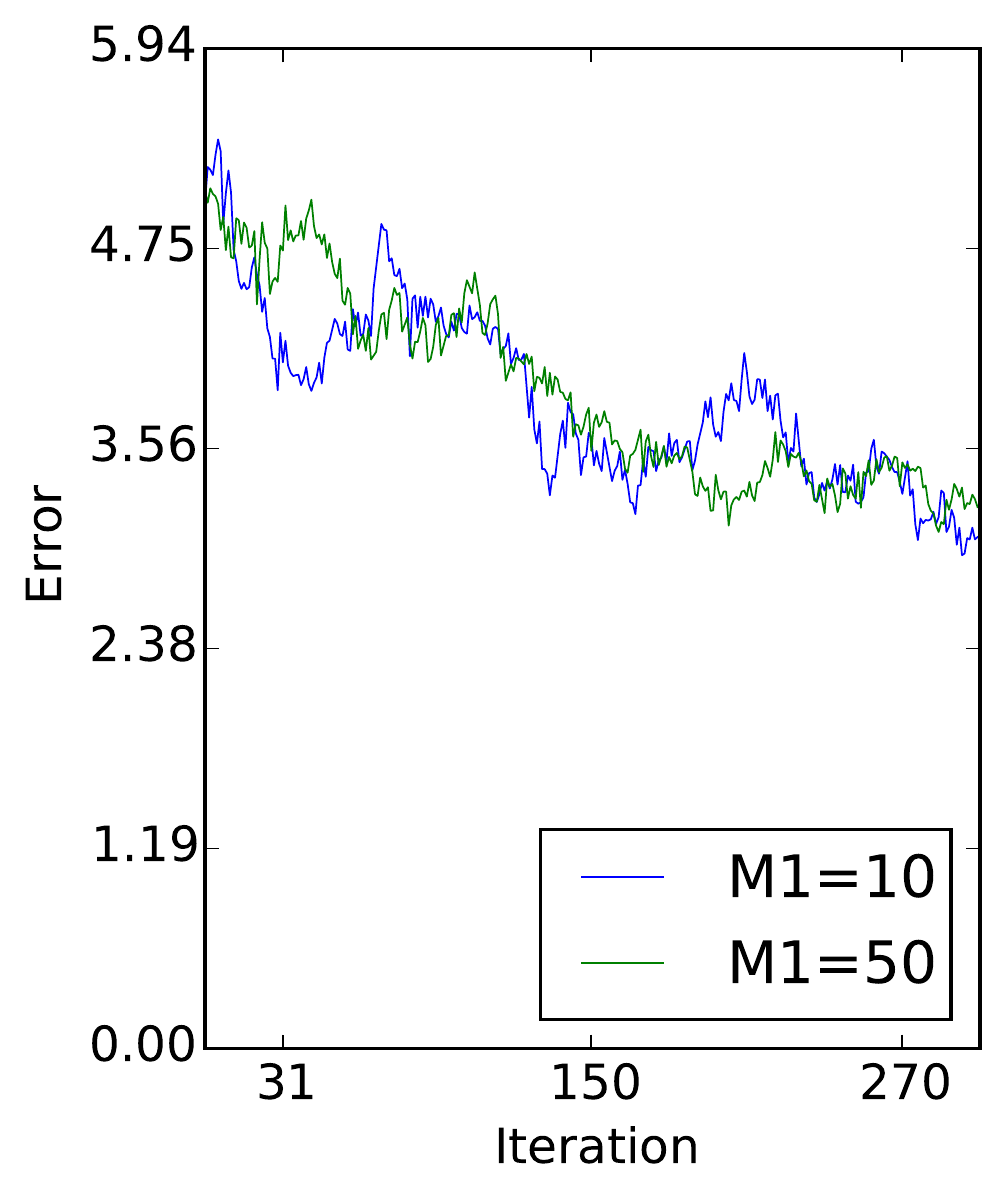}
\end{center}
\caption{Error trajectories of MVD-based optimization for different values of $M^1$.\label{fig:mvderr}}
\end{figure}
We have implemented Algorithm \ref{mvd-der-little} to estimate gradients as part of  a neural network training procedure. The network was trained to classify digits from the MNIST dataset\footnote{http://yann.lecun.com/exdb/mnist/}. The network was trained using the first 5000 images digits of the dataset. We report how the empirical error evolved during training.

Formally, our objective function is the sum of $k=5000$ terms $J^{1},J^{2},\hdots, J^{k}$, one for each image.
For simplicity, let us consider the term for a single image.
In this case, each training datum consists of a $28\times 28$ pixel image $v \in \{0,1\}^{784}$ and a label $\ell \in \{0,1\}^{10}$.
The vector $\ell$ is an indicator vector, with a single entry equal to one in the position corresponding to the class of the image and zeros elsewhere. The first $784$ nodes of the network receive the image pixels as an external input. The last ten nodes are output units.
  Hence the network has $784+10=794$ units (e.g. $n=794$). Among the units we allow arbitrary connectivity, resulting in $794\times 794$ connections. 
 The error function $e$ is $$e(x) = \sum\limits_{i=1}^{10}|x_{784+i} - \ell_i|$$ and the stationary cost is 
 $J_{\theta} = \mathbb{E}_{x\sim\pi_{\theta}}[e(x)]$.
 Note that $\pi_{\theta}$ depends on the parameters $\theta$ and also the external input $v$. 

In words, the goal of training is to find the parameters $\theta$ which guarantee that when the image $v$ is fixed as the input to the network, the network will compute the class of the image and indicate this class at the output nodes.

Each weight $w_{i,j}$ is initialized with a sample from a uniform distribution on $[-0.01,0.01]$. The biases $b_i$ are initialized with a sample from the same distribution. The parameters $\theta_{n+1}$ follow the recursion

$$\theta_{n+1} = \theta_n - \epsilon \Delta_{n}$$
where at time $n$, the perturbation $\Delta_n$ is computed via Algorithm \ref{mvd-der-little} as described above.

The parameters of Algorithm \ref{mvd-der-little} were set up as follows. The value of $M^0$ was 10 and we experimented with two possible values for $M^1$, either $M^1=10$ or $M^1=50$. Common random numbers were used in the routines for updating $x^+$ and $x^-$ for variance reduction. The algorithm ran for 30000 parameter updates and the empirical error is reported every 500 updates.
Results are shown in Figure~\ref{fig:mvderr}. The trend is similar but using $M^1= 50$ iterations inside Algorithm \ref{mvd-der-little} seems to have less variance.

\section{Discussion}
In this {\color{black} paper} we studied gradient estimation for the Little model.  Since a closed form solution for the stationary distribution is not known, methods which work for similar models such as the Boltzmann machine or sigmoid belief network cannot be used.
 To address this we introduced a method to calculate derivatives in the Little model based on measure-valued differentiation. To get a gradient estimate this way, one has to run the Little model for some time to get an initial state, then generate a random direction, and then run two Markov kernels $P_{\theta,v}^+$ and $P_{\theta,v}^{-}$ and use the observed errors in both chains to form the gradient estimate. There are several parameters of the method - the $M^0$ and $M^1$ of Algorithm \ref{mvd-der-little}. Future work will study the dependence of the algorithm performance, such as bias and variance, on these parameters. A more detailed numerical study will also aid in tuning these parameters. The authors believe the general idea of pairing random directions with measure-valued-differentiation could enable optimization in other models as well.

%
%


\bibliographystyle{spmpsci}      
\bibliography{super}   


\appendix
\section{Derivations Related to the Little Model}
\subsection{Derivation of Equation \ref{mvd-eqn}}\label{mvd-der-eqn}
Fix an $x^0$ and a direction 
$v \in \mathbb{R}^{n\times n}\times\mathbb{R}^n$. 
Then 
\begin{align*}
\nabla_{\lambda}P_{\theta + \lambda v}(x^{0},x^{1}) 
&= 
P_{\theta + \lambda v}(x^0,x^1)
\nabla_{\lambda}\log P_{\theta + \lambda v}(x^0,x^1)  \\
&=
P_{\theta + \lambda v}(x^0,x^1)
\sum\limits_{i=1}^{n}\nabla_{\lambda}
\log\left(\sigma( (x_{i}^{1})^{\dag}u_{i}(x^{0},\theta+\lambda v))\right)  \\
&= P_{\theta + \lambda v}(x^0,x^1)
\sum\limits_{i=1}^{n}
\left(1 -\sigma( (x_{i}^{1})^{\dag}u_{i}(x^{0},\theta+\lambda v))\right)
  (x_{i}^{1})^{\dag}
\nabla_{\lambda}u_{i}(x^{0},\theta+\lambda v) \\
&= 
P_{\theta + \lambda v}(x^0,x^1)
\sum\limits_{i=1}^{n}
\left(1 -\sigma( (x_{i}^{1})^{\dag}u_{i}(x^{0},\theta+\lambda v))\right)
(x_{i}^{1})^{\dag}
\left(\sum\limits_{j=1}^{n}v_{i,j}x^{0}_{j} + v_{i}\right).
\end{align*}
Evaluating this at $\delta=0$ we find that
\begin{equation}\label{delta-0-id}
\nabla_{\theta}P_{\theta}(x^{0},x^{1})v 
= P_{\theta }(x^0,x^1)
\sum\limits_{i=1}^{n}(1 -\sigma( (x_{i}^{1})^{\dag}u_{i}(x^{0},\theta)))
(x_{i}^{1})^{\dag}
\left(\sum\limits_{j=1}^{n}v_{i,j}x^{0}_{j}  + v_{i}\right).
\end{equation}
Note also that 
$$
(1- \sigma( x^{\dag}u))x^{\dag} = \begin{cases} 
(1-\sigma(u)) &\text{ if } x = 1   \\
-\sigma(u)&\text{ if } x = 0 
\end{cases}
$$
which means
\begin{equation}\label{a-sig-id}
(1-\sigma(x^{\dag}u))x^{\dag} = x - \sigma(u).
\end{equation}
Combining (\ref{delta-0-id}) and (\ref{a-sig-id}),
\begin{equation}
\begin{split}
&\nabla_{\theta}\textstyle\sum\limits_{x^{1}}e(x^{1})P_{\theta}(x^{0},x^{1})v 
\\&= 
\ts\sum\limits_{x^{1}}e(x^{1})P_{\theta}(x^0,x^1)
\ts\sum\limits_{i=1}^{n}(1 -\sigma( (x_{i}^{1})^{\dag}u_{i}(x^{0},\theta)))(x_{i}^{1})^{\dag}
\left(\ts\sum\limits_{j=1}^{n}v_{i,j}x^{0}_{j} + v_i\right)\\
&= 
\ts\sum\limits_{x^{1}}e(x^{1})P_{\theta}(x^0,x^1)
\ts\sum\limits_{i=1}^{n}(x_{i}^{1} -\sigma(u_{i}(x^{0},\theta)))
\left(\ts\sum\limits_{j=1}^{n}v_{i,j}x^{0}_{j} + v_i\right)\\
&=
\ts\sum\limits_{x^{1}}e(x^{1})P_{\theta}(x^0,x^1)\left[
\ts\sum\limits_{i=1}^{n}x_{i}^{1}
\left(\ts\sum\limits_{j=1}^{n}v_{i,j}x^{0}_{j} + v_i\right) 
-
\ts\sum\limits_{i=1}^{n}\sigma(u_{i}(x^{0},\theta))
\left(\ts\sum\limits_{j=1}^{n}v_{i,j}x^{0}_{j} + v_i\right)\right].
\end{split}
\end{equation}
Splitting each $v_{i,j}$ and $v_i$ into positive and negative parts,
\begin{equation}\label{the-big-one}
\begin{split}
&=
\ts\sum\limits_{x^{1}}e(x^{1})P_{\theta}(x^0,x^1)\left[
\ts\sum\limits_{i=1}^{n}x_{i}^{1}
                   \left(\ts\sum\limits_{j=1}^{n}(v_{i,j})_{+}x^{0}_{j} + (v_{i})_+\right)
-
\ts\sum\limits_{i=1}^{n}x_{i}^{1}
                 \left(\ts\sum\limits_{j=1}^{n}(v_{i,j})_{-}x^{0}_{j} + (v_{i})_{-}\right)\right]\\
&
-
\ts\sum\limits_{x^{1}}e(x^{1})P_{\theta}(x^0,x^1)\left[
\ts\sum\limits_{i=1}^{n}\sigma(u_{i}(x^{0},\theta))
\left(\ts\sum\limits_{j=1}^{n}(v_{i,j})_{+}x^{0}_{j} + (v_{i})_{+}\right)
-
\ts\sum\limits_{i=1}^{n}\sigma(u_{i}(x^{0},\theta))
\left(\ts\sum\limits_{j=1}^{n}(v_{i,j})_{-}x^{0}_{j} + (v_{i})_{-}\right)\right] \\
&=
\left(
  \ts\sum\limits_{x^{1}}e(x^{1})P_{\theta}(x^0,x^1)\left[
  \ts\sum\limits_{i=1}^{n}x_{i}^{1}
  \left(\ts\sum\limits_{j=1}^{n}(v_{i,j})_{+}x^{0}_{j} + (v_i)_{+}\right)
  +
  \ts\sum\limits_{i=1}^{n}\sigma(u_{i}(x^{0},\theta))
  \left(\ts\sum\limits_{j=1}^{n}(v_{i,j})_{-}x^{0}_{j} + (v_{i})_{-}\right) 
\right]\right)
\\&\quad-
\left(
  \ts\sum\limits_{x^{1}}e(x^{1})P_{\theta}(x^0,x^1)\left[
  \ts\sum\limits_{i=1}^{n}x_{i}^{1}
  \left(\sum\limits_{j=1}^{n}(v_{i,j})_{-}x^{0}_{j} + (v_i)_{-}\right)
  +
  \ts\sum\limits_{i=1}^{n}\sigma(u_{i}(x^{0},\theta))
  \left(\ts\sum\limits_{j=1}^{n}(v_{i,j})_{+}x^{0}_{j} + (v_{i})_{+}\right)\right]
\right) \\
&=
\ts\sum\limits_{x^{1}}e(x^{1})P_{\theta}(x^0,x^1)\sum\limits_{i=1}^{n}
\left[
  x_{i}^{1}\left(\ts\sum\limits_{j=1}^{n}(v_{i,j})_{+}x^{0}_{j} + (v_{i})_{+}\right)
  +
  \sigma(u_{i}(x^{0},\theta))
  \left(\ts\sum\limits_{j=1}^{n}(v_{i,j})_{-}x^{0}_{j} + (v_i)_{-} \right) 
\right]\\
&\quad-
\ts\sum\limits_{x^{1}}e(x^{1})P_{\theta}(x^0,x^1)\ts\sum\limits_{i=1}^{n}
\left[
  x_{i}^{1}\left(\sum\limits_{j=1}^{n}(v_{i,j})_{-}x^{0}_{j} + (v_i)_-\right)
  +
  \sigma(u_{i}(x^{0},\theta))\left(\ts\sum\limits_{j=1}^{n}(v_{i,j})_{+}x^{0}_{j} + (v_i)_{+}\right)
\right].
\end{split}
\end{equation}
Note that
\begin{equation}\label{c-confirm}
\begin{split}
&\sum\limits_{x^{1}}P_{\theta}(x^0,x^1)\sum\limits_{i=1}^{n}
\left[
x_{i}^{1}\left(\sum\limits_{j=1}^{n}(v_{i,j})_{+}x^{0}_{j} + (v_{i})_{+}\right)
+
\sigma(u_{i}(x^{0},\theta))
\left(\sum\limits_{j=1}^{n}(v_{i,j})_{-}x^{0}_{j} + (v_{i})_{-}\right) \right] \\
&=\sum\limits_{i=1}^{n}\left(\sum\limits_{x^{1}}P_{\theta}(x^0,x^1)x_{i}^{1}\right)
\left(\sum\limits_{j=1}^{n}(v_{i,j})_{+}x^{0}_{j} + (v_{i})_{+}\right)
+
\sum\limits_{i=1}^{n}\sigma(u_{i}(x^{0},\theta))
\left(\sum\limits_{j=1}^{n}(v_{i,j})_{-}x^{0}_{j} + (v_i)_{-}\right) \\
&=\sum\limits_{i=1}^{n}\sigma(u_{i}(x^{0},\theta))
\left(\sum\limits_{j=1}^{n}(v_{i,j})_{+}x^{0}_{j} + (v_i)_{+}\right)
+
\sum\limits_{i=1}^{n}\sigma(u_{i}(x^{0},\theta))
\left(\sum\limits_{j=1}^{n}(v_{i,j})_{-}x^{0}_{j} + (v_{i})_{-}\right) \\
&=
\sum\limits_{i=1}^{n}\sigma(u_{i}(x^0))|v_i| + 
\sum\limits_{i=1}^{n}\sum\limits_{j=1}^{n}
\sigma(u_{i}(x^{0},\theta))|v_{i,j}|x^{0}_{j}.
\end{split}
\end{equation}
Combining (\ref{the-big-one}) with (\ref{c-confirm}) and the definitions (\ref{little-pplus-def}) and (\ref{little-pminus-def}) we obtain (\ref{mvd-eqn}).
\subsection{Derivation of Equations \ref{little-p-one} and \ref{little-p-gen}}\label{mvd-p}
We have
\begin{equation}\label{bzb0}
\begin{split}
Q(x_1) 
&= 
\sum\limits_{x_2 \in \{0,1\}}\hdots\sum\limits_{x_n \in \{0,1\}}
Q(x_1,x_2,\hdots,x_n) \\
&=
\sum\limits_{x_2 \in \{0,1\}}\hdots\sum\limits_{x_n \in \{0,1\}}
\frac{1}{c}
\prod_{i=1}^{n}
\beta_i^{x_i}
(1-\beta_i)^{1-x_i}
\left(d + \sum\limits_{i=1}^{n}\alpha_{i}x_i \right)\\
&=
\sum\limits_{x_2 \in \{0,1\}}\hdots\sum\limits_{x_n \in \{0,1\}}
\frac{1}{c}
\prod_{i=2}^{n}\beta_i^{x_i}(1-\beta_i)^{1-x_i}
\beta_1^{x_1}(1-\beta_1)^{1-x_{1}}
\left(d + \alpha_1x_1 + \sum\limits_{i=2}^{n}\alpha_{i}x_i \right)\\
&=
\beta_1^{x_1}(1-\beta_1)^{1-x_1}
\sum\limits_{x_2 \in \{0,1\}}\hdots\sum\limits_{x_n \in \{0,1\}}
\frac{1}{c}
\prod_{i=2}^{n}\beta_i^{x_i}(1-\beta_i)^{1-x_i}
\left(d + \alpha_1x_1 + \sum\limits_{i=2}^{n}\alpha_{i}x_i \right)\\
&=
\beta_1^{x_1}(1-\beta_1)^{1-x_1}
\sum\limits_{x_2 \in \{0,1\}}\hdots\sum\limits_{x_n \in \{0,1\}}
\frac{1}{c}\prod_{i=2}^{n}\beta_i^{x_i}(1-\beta_i)^{1-x_i}
\left(d + \sum\limits_{i=2}^{n}\alpha_{i}x_i \right) \\
&\quad\quad+
\beta_1^{x_1}(1-\beta_1)^{1-x_1}
\sum\limits_{x_2 \in \{0,1\}}\hdots\sum\limits_{x_n \in \{0,1\}}
\frac{1}{c}\prod_{i=2}^{n}\beta_i^{x_i}(1-\beta_i)^{1-x_i}\alpha_1x_1 \\
&= 
\beta_1^{x_1}(1-\beta_1)^{1-x_1}
\alpha_1x_1\frac{1}{c} \\&\quad\quad+
\beta_1^{x_1}(1-\beta_1)^{1-x_1}
\frac{1}{c}
\sum\limits_{x_2 \in \{0,1\}}\hdots\sum\limits_{x_n \in \{0,1\}}
\prod_{i=2}^{n}\beta_i^{x_i}(1-\beta_i)^{1-x_i}\left(d + \sum\limits_{i=2}^{n}\alpha_{i}x_i \right).
\end{split}
\end{equation}
To simplify this equation, note that for $n>1$,
\begin{equation}\label{bzb1}
\begin{split}
&\sum\limits_{x_1 \in \{0,1\}}\hdots\sum\limits_{x_n \in \{0,1\}}
\prod_{i=1}^{n}\beta_i^{x_i}(1-\beta_i)^{1-x_i}
\left( d + \sum\limits_{i=1}^{n}a_ix_i\right) \\&= 
\sum\limits_{x_2 \in \{0,1\}}\hdots\sum\limits_{x_n \in \{0,1\}}\bigg[
\beta_1\prod_{i=2}^{n}\beta_i^{x_i}(1-\beta_i)^{1-x_i}
\left( d + \sum\limits_{i=2}^na_ix_i + a_1\right) 
\\&\quad\quad\quad\quad\quad\quad\quad\quad\quad+ 
(1-\beta_1)\prod_{i=2}^{n}\beta_i^{x_i}(1-\beta_i)^{1-x_i}
\left(d + \sum\limits_{i=2}^na_ix_i\right)\bigg] \\
&=\sum\limits_{x_2 \in \{0,1\}}\hdots\sum\limits_{x_n \in \{0,1\}}\bigg[
\beta_1a_1\prod_{i=2}^{n}\beta_i^{x_i}(1-\beta_i)^{1-x_i}
+ 
\prod_{i=2}^{n}\beta_i^{x_i}(1-\beta_i)^{1-x_i}
\left(d + \sum\limits_{i=2}^na_ix_i\right) \bigg]\\
\\&= 
\beta_1\alpha_1 
+ 
\sum\limits_{x_2 \in \{0,1\}}\hdots\sum\limits_{x_n \in \{0,1\}}
\prod_{i=2}^{n}\beta_i^{x_i}(1-\beta_i)^{1-x_i}
\left(d + \sum\limits_{i=2}^na_ix_i\right), 
\end{split}
\end{equation}
and if $n=1$ then
\begin{equation}\label{bzb2}
\begin{split}
\sum\limits_{x_1 \in \{0,1\}}\prod_{i=1}^{n}\beta_i^{x_i}(1-\beta_i)^{1-x_i}
\left(d + \sum\limits_{i=1}^{n}a_ix_i\right) 
&= \beta_1(d + a_1) + (1-\beta_1)d  \\
&= 
\beta_1d + \beta_1a_1 + d - \beta_1d = \beta_1a_1 +d.
\end{split}\end{equation}
Combining Equation \ref{bzb1} and Equation \ref{bzb2},  we see that for any $n\geq 1$,
\begin{equation}\label{bzb3}
\sum\limits_{x_1 \in \{0,1\}}\hdots\sum\limits_{x_n \in \{0,1\}}
\prod_{i=1}^{n}\beta_i^{x_i}(1-\beta_i)^{1-x_i}
\left( d + \sum\limits_{i=1}^{n}a_ix_i \right) 
= d + \sum\limits_{i=1}^{n}\beta_i\alpha_i.
\end{equation}
Combining Equation \ref{bzb0} with Equation \ref{bzb3},
\begin{align*}
Q(x_1) 
&= 
\beta_1^{x_1}(1-\beta_1)^{1-x_1}\frac{\alpha_1x_1}{c} 
+ 
\beta_1^{x_1}(1-\beta_1)^{1-x_1}\frac{1}{c}
\left(d + \sum\limits_{i=2}^{n}\beta_i\alpha_i\right) \\
&= 
\beta_1^{x_1}(1-\beta_1)^{1-x_1}\frac{1}{c}
\left[d + \alpha_1x_1 + \sum\limits_{i=2}^{n}\beta_i\alpha_i\right].
\end{align*}
In general,
\begin{align*}
&Q(x_k, x_{k-1},\hdots, x_1)  \\
&= 
\ts\sum\limits_{x_{k+1} \in \{0,1\}}\hdots\sum\limits_{x_n \in \{0,1\}}
Q(x_1,\hdots,x_k,,\hdots,x_n) \\
&=
\ts\sum\limits_{x_{k+1} \in \{0,1\}}\hdots\sum\limits_{x_n \in \{0,1\}}
\frac{1}{c}\prod_{i=1}^{n}\beta_i^{x_i}(1-\beta_i)^{1-x_i}
\left( d + \sum\limits_{i=1}^{n}a_{i}x_i\right) \\
&=
\ts\sum\limits_{x_{k+1} \in \{0,1\}}\hdots\sum\limits_{x_n \in \{0,1\}}
\frac{1}{c}
\prod_{i=k+1}^{n}\beta_i^{x_i}(1-\beta_i)^{1-x_i} \\
&\quad\quad\times
\beta_k^{x_k}(1-\beta_k)^{1-x_k}
\prod_{i=1}^{k-1}\beta_i^{x_i}(1-\beta_i)^{1-x_i}
\left(
   \sum\limits_{i=1}^{k-1}a_ix_i + a_kx_k + d + \sum\limits_{i=k+1}^{n}a_{i}x_i 
   \right) \\
&=
\ts\sum\limits_{x_{k+1} \in \{0,1\}}\hdots\sum\limits_{x_n \in \{0,1\}}
\frac{1}{c}
\prod_{i=k+1}^{n}\beta_i^{x_i}(1-\beta_i)^{1-x_i}
\beta_k^{x_k}(1-\beta_k)^{1-x_k}
\prod_{i=1}^{k-1}\beta_i^{x_i}(1-\beta_i)^{1-x_i}
\left(a_kx_k + \sum\limits_{i=1}^{k-1}a_ix_i\right)  \\
&\quad+
\ts\sum\limits_{x_{k+1} \in \{0,1\}}\hdots\sum\limits_{x_n \in \{0,1\}}
\frac{1}{c}
\prod_{i=k+1}^{n}\beta_i^{x_i}(1-\beta_i)^{1-x_i}
\beta_k^{x_k}(1-\beta_k)^{1-x_k}
\prod_{i=1}^{k-1}\beta_i^{x_i}(1-\beta_i)^{1-x_i}
\left(d + \sum\limits_{i=k+1}^{n}a_{i}x_i \right)
 \\
&=
\ts\beta_k^{x_k}(1-\beta_k)^{1-x_k}\frac{1}{c}
\prod_{i=1}^{k-1}\beta_i^{x_i}(1-\beta_i)^{1-x_i}
\left(a_kx_k + \sum\limits_{i=1}^{k-1}a_ix_i\right)  \\
&\quad+
\ts\beta_k^{x_k}(1-\beta_k)^{1-x_k}\frac{1}{c}
\prod_{i=1}^{k-1}\beta_i^{x_i}(1-\beta_i)^{1-x_i}
\sum\limits_{x_{k+1} \in \{0,1\}}\hdots\sum\limits_{x_n \in \{0,1\}}
\prod_{i=k+1}^{n}\beta_i^{x_i}(1-\beta_i)^{1-x_i}
\left(d + \sum\limits_{i=k+1}^{n}a_{i}x_i \right) \\
&=
\ts\beta_k^{x_k}(1-\beta_k)^{1-x_k}\frac{1}{c}
\prod_{i=1}^{k-1}\beta_i^{x_i}(1-\beta_i)^{1-x_i}
\left(a_kx_k + \sum\limits_{i=1}^{k-1}a_ix_i\right)  \\
&\quad+
\ts\beta_k^{x_k}(1-\beta_k)^{1-x_k}\frac{1}{c}
\prod_{i=1}^{k-1}\beta_i^{x_i}(1-\beta_i)^{1-x_i}
\left( d +\sum\limits_{i=k+1}^{n}\beta_i\alpha_i\right) \\
&=
\ts\frac{1}{c}
\prod_{i=1}^{k}\beta_i^{x_i}(1-\beta_i)^{1-x_i}
\left(d+ \sum\limits_{i=1}^{k}a_ix_i + \sum\limits_{i=k+1}^{n}\beta_i\alpha_i\right).
\end{align*}

\end{document}